\documentclass[10.5pt]{amsart}

\usepackage{geometry}              
\usepackage{graphicx}
\usepackage{amsmath,amsthm,amssymb}
\usepackage{epstopdf}
\usepackage{url}

\usepackage{amssymb,amsfonts}

\usepackage[all,arc]{xy}
\usepackage{enumerate}
\usepackage{mathrsfs}
\usepackage{amsthm}
\usepackage{standalone}
\usepackage{tikz,pgf}
\usepackage{graphicx}
\usepackage{caption}
\usepackage{pinlabel}

 \usepackage{amsaddr}

\usepackage{array}
\usepackage{enumitem}
\usepackage{scrextend}
\usepackage{tikz}
\usepackage{tikz-cd}
\usepackage{amsmath}
\usepackage{permute}
\usetikzlibrary{positioning}

\usepackage{clrscode3e}

\newtheorem{thm}{Theorem}[section]

\newtheorem{prop}[thm]{Proposition}
\newtheorem{lem}[thm]{Lemma}
\newtheorem{conj}[thm]{Conjecture}

\newtheorem*{rem*}{Remark}
\theoremstyle{definition}
\newtheorem{defn}[thm]{Definition}

\newtheorem{example}[thm]{Example}

\newtheorem{rem}[thm]{Remark}

\numberwithin{equation}{section}

\title{Coxeter quotients of knot groups through
16 crossings}
\thanks{MSC codes: 57M25, 57M27, 57M05.}

\author{Ryan Blair, Alexandra Kjuchukova, Nathaniel Morrison}

\begin{document}
\maketitle

\begin{abstract} We find explicit maximal rank Coxeter quotients for the knot groups of $595,515$ out of the $1,701,936$ knots through 16 crossings. We thus calculate the bridge numbers and verify Cappell and Shaneson's Meridional Rank Conjecture for these knots. In addition, we provide a computational tool for establishing the conjecture for those knots beyond 16 crossings whose meridional ranks can be detected via finite Coxeter quotients.

 \end{abstract}

\section{Introduction}

Knot groups and their quotients provide effective techniques for distinguishing and tabulating knots, studying their properties and calculating a variety of classical invariants. Prime knots are determined, up to reflection, by their groups~\cite{gordon1989knots}. Further, dihedral and symmetric group quotients have been as instrumental as polynomial invariants in creating and expanding the knot table~\cite{HTW, perko1974classification}. In this paper, we adopt a computational approach to studying two notoriously elusive knot invariants: the bridge number and meridional rank. We perform an exhaustive search, covering the groups of tabulated knots through 16 crossings, for quotients onto finite Coxeter groups. We find $595,515$ quotients for knots of bridge number at least 3, which implies that $601,061$ out of the first 1,701,936 (non-cyclic) knot groups admit maximal rank quotients, in the sense defined below, onto finite Coxeter groups. For approximately $38\%$ 
of these knots, we compute the bridge number for the first time. Our findings are summarized in Section~\ref{section-tables}.

Recall that given a Coxeter presentation for a Coxeter group $G$, a {\it reflection} is any element conjugate to one of the generators in this presentation. The {Coxeter rank} of $G$ is the cardinality of a minimal generating set of reflections for $G$. In this paper, the Coxeter rank will be denoted by $r(G)$ and may also be called simply ``the rank of $G$". Whenever we consider a group homomorphism $\rho: \pi_1(S^3\backslash K)\twoheadrightarrow G$ from a knot group onto a Coxeter group $G$, we will always assume that meridians of $K$ map to reflections in $G$. Sometimes we will emphasize this property by saying that $\rho$ is a {\it good} quotient.

Consider a good quotient  $\rho: \pi_1(S^3\backslash K)\twoheadrightarrow G$ as above. If $r(G)$ equals the bridge number of $K$, we say that $\rho$ is a {\it maximal rank} Coxeter quotient, abbreviated MRCQ. As the phrase suggests, the Coxeter rank of a good quotient for $K$ can never exceed the bridge number $\beta(K)$. Indeed, recall that $\beta(K)$ is an upper bound for the meridional rank $\mu(K)$. Furthermore, a generating set of meridians is mapped by a good quotient map to a generating set of reflections. Hence, for any good quotient map $\varphi: \pi_1(S^3\backslash K)\twoheadrightarrow G$, we have the inequalities
\begin{equation} \label{ineqs1}
	 \beta(K)\geq \mu(K)\geq r(G).
\end{equation}
Thus, we have a MRCQ precisely when $\beta(K)= r(G)$ holds, and this equality can sometimes be verified diagrammatically.

\begin{prop}\label{prop-equalities}\cite{baader2021coxeter}
	Let $D$ be a diagram for a knot $K$. Denote by $\omega(D)$ the Wirtinger number (Definition~\ref{def-wirt})  of $D$. Assume that $G$ is a Coxeter group such that there exists a good quotient $\pi_1(S^3\backslash K)\twoheadrightarrow G$. If the Coxeter rank of $G$ satisfies $r(G)=\omega(D)$, the Meridional Rank Conjecture holds for $K$ and we have $$\omega(D)=\omega(K)=\beta(K)=\mu(K)=r(G).$$
\end{prop}
\begin{proof}
The result follows from Equation~\ref{ineqs1}, combined with the fact that the Wirtinger number of any diagram of $K$ is an upper bound for the bridge number: $\omega(D)\geq\omega(K)=\beta(K)$, which is proved in~\cite{blair2020wirtinger}.	
\end{proof}

Given an knot $K$ with diagram $D$, we say that $D$ exhibits a maximal rank Coxeter quotient if there exists a good quotient $\varphi: \pi_1(S^3\backslash K)\twoheadrightarrow G$ such that $r(G)=\omega(D)$. The existence of such a $\varphi$ allows us to apply Proposition~\ref{prop-equalities} to prove the Meridional Rank Conjecture (Kirby List~\cite{kirby1995problems}, Problem~1.11) for $K$. Moreover, $D$ realizes the Wirtinger number of $K$, that is,  $\omega(D)$ equals the bridge number $\beta(K)$. 
In this work, we determine the diagrams in the Hoste-Thistlethwaite-Weeks table~\cite{HTW} through 16 crossings which exhibit maximal rank quotients onto finite Coxeter groups. We thereby compute the meridional ranks and bridge numbers for the corresponding knots, along the way showing that these knots satisfy the Meridional Rank Conjecture of Cappell and Shaneson. Note that the conjecture has been proven in a variety of special cases, notably torus links~\cite{RZ87}, links of meridional rank two~\cite{BZ89}, Montesinos links~\cite{BZ85} and generalized Montesinos links~\cite{LM93}, twisted links~\cite{baader2021coxeter}, and certain classes of arborescent links ~\cite{baader2021coxeter, baader2020twigs}, among others \cite{BDJW, BJW, CH14, baader2019symmetric}. It is unknown how many and precisely which knots through 16 crossings are covered by one or more of these theoretical results. In practice, however, it can be challenging to determine whether a given knot satisfies the hypotheses of some of the theorems cited above, particularly when these hypotheses include the existence of a diagram with special properties. 
This makes it difficult to identify potential counter-examples to the conjecture, that is, knots which do not belong to any of the special cases for which the conjecture is known to hold. 
Our work is a step toward bridging this gap.  
Moreover, when the meridional rank of a knot $K$ is detected by a finite Coxeter quotient, we explicitly compute the bridge number and meridional rank of $K$ from Gauss code for a diagram of $K$. 

\begin{thm}[Main Theorem]\label{thm-main}
Let $D$ be a knot diagram. $D$ admits a maximal rank quotient onto a finite Coxeter group $H$ if and only if such a quotient is detected by the algorithm outlined in Section \ref{section-homsearch}. 
\end{thm}

The result follows from three main ingredients: the equality between the bridge number and Wirtinger number of a knot (Theorem~\ref{thm-wirt}); the easy fact that the existence of a Coxeter quotient of a knot group can be detected in {\it any} diagram of the knot (Proposition~\ref{lem-AnyDiagram}); and the celebrated classification of finite Coxeter groups (Theorem~\ref{CoxeterClassification}). These results are recalled in Section~\ref{section-background}, and our proof appears in Section~\ref{section-homsearch}, which is dedicated to showing that the homomorphism search we perform is exhaustive. Therein, we also describe our method for trimming the set of possible generating sets for finite Coxeter groups without compromising the exhaustiveness of the search; this step was necessary in order to make the computation feasible.  
We have implemented the algorithm and run it on all knots through 16 crossings. Our search identified all diagrams in the knot table which admit MRCQs onto finite Coxeter groups. The data obtained by running our algorithm for all tabulated knots through 16 crossings is summarized in Section~\ref{section-tables}. We conjecture that crossing number minimizing diagrams of prime knots through 16 crossings realize the Wirtinger numbers of the corresponding knots, that is, we posit that $\omega(D)=\omega(K)$ for any minimal diagram $D$ of a prime knot $K$ through 16 crossings. If true, this would imply that we have identified precisely the knots in the table which admit maximal rank quotients onto finite Coxeter groups.

\subsection{Applications.}
The values of bridge number and meridional rank established in this paper have implications for other difficult to evaluate knot invariants. An early version of our computation was used to find the stick number of knots in several challenging cases~\cite{blair2020knots}. Additionally, the bridge number gives a lower bound on the superbridge index~\cite{Kuiper87}. It is possible to use the values of bridge number established by our algorithm together with recent upper bounds on superbridge index ~\cite{Shonkwiler20} to compute the superbridge index of some knots for which the value was previously unknown. Further, our algorithm can be adapted to accept non-planar Gauss codes and thus to give lower bounds on the virtual bridge number of virtual knots. When paired with the upper bounds from~\cite{pongtanapaisan2019wirtinger}, this technique can be used compute the virtual bridge number of many virtual knots. Finally, our algorithm can be used in conjunction with the results in~\cite{joseph2021bridge} to establish the meridional rank and bridge number of certain twist spun knots, which are knotted 2-spheres in~$\mathbb{R}^4$.

\section{Bridge numbers via Coxeter quotients}
\label{section-background}

We recall some basic definitions and necessary background on knot colorings, Coxeter groups and Wirtinger numbers of links, as well as the approach from~\cite{baader2021coxeter} for computing bridge numbers using Coxeter quotients of knot groups. 

\subsection{Classification of finite Coxeter groups} \label{section-coxeter review}\label{Coxeterbackground}

\begin{defn}
Let $\Gamma$ be a finite simple graph with edges labeled by intergers greater than $1$. The Coxeter group $C(\Gamma)$ is a group generated by a set in bijective correspondence with the vertices of $\Gamma$, subject to the following two types of relations:

\begin{enumerate}
\item $s^2=1$ for all generators s.
\item $(st)^k=1$ for all pairs of generators $s,t$ connected by and edge of weight $k\in \mathbb{N}$. 
\end{enumerate}
We call $\Gamma$ the Coxeter graph and the presentation determined by $\Gamma$ a Coxeter presentation of $C(\Gamma)$.

\end{defn}

Note that a given Coxeter group can have multiple Coxeter presentations. See Figure~\ref{fig-D6}. 
However, as we will see, this is not the case for the class of Coxeter groups that we use in this paper, namely finite Coxeter groups in which all reflections belong to the same conjugacy class.

\begin{figure}
\labellist
 \pinlabel $x$ at 84 65
 \pinlabel $3$ at 140 100
 \pinlabel $y^3$ at 140 200
 \pinlabel $2$ at 100 140 
  \pinlabel $2$ at 178 140 
 \pinlabel $xy^2$ at 200 65
 \pinlabel $x$ at 340 65
 \pinlabel $6$ at 400 100
 \pinlabel $xy$ at 460 65
\endlabellist
\includegraphics[width=80mm]{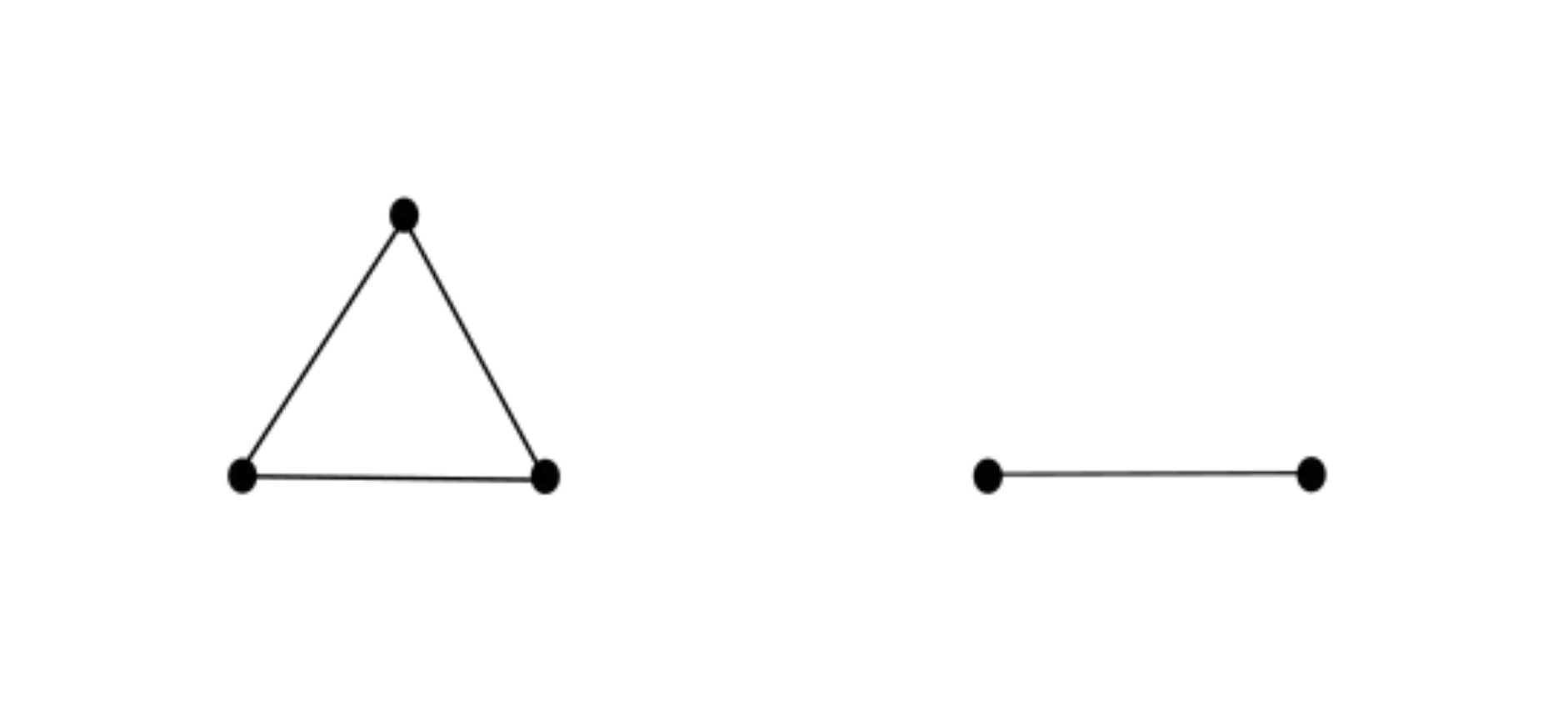}
\caption{Two Coxeter presentations for the dihedral group $D_6=\langle x, y | x^2=y^6= 1, xyx=y^{-1}\rangle $} \label{fig-D6}
\end{figure}
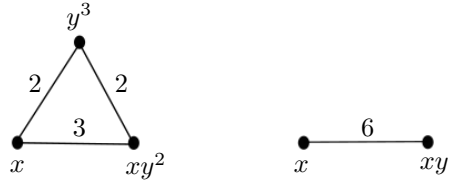

Finite Coxeter groups are classified. Every finite Coxeter group is isomorphic to exactly one group in a list of four infinite families and six exceptional groups. For more details on this well-known classification theorem, see any of the standard Coxeter group references, for instance Bourbaki~\cite{Bourbaki}. 

\begin{thm}\label{CoxeterClassification}
Every finite Coxeter group is a finite reflection group. Moreover, every finite Coxeter group is isomorphic to exactly one of the groups $A_n (n \geq 1), B_n (n \geq 2), D_n (n \geq 4), E_6,
E_7, E_8, F_4, H_3, H_4$ and $I_2(m) (m \geq 2)$.
\end{thm}

Note that each of the groups listed in the previous theorem are defined by a Coxeter graph and corresponding Coxeter presentation.

Given a Coxeter presentaion of a group $C(\Gamma)$, the Coxeter rank $r(C(\Gamma))$ is known to equal the number of vertices of $\Gamma$, see for example Lemma 2.1 in~\cite{felikson2010reflection}. In this paper, we are interested in all Coxeter presentations of finite Coxeter groups with Coxeter rank 3, 4, or 5, such that the set of reflections form a single conjugacy class. On the one hand, it was shown in~\cite{blair2020wirtinger} that all knots in the knot table up to 16 crossings have bridge number at most 5. On the other, knots of bridge number less than $3$ are classically known to admit MRCQ; indeed all two-bridge knots have dihedral quotients. 
Therefore, the Coxeter groups needed in order to find MRCQs for the remaining cases are those with Coxeter rank 3, 4, or 5. Additionally, all meridians in the fundamental group of a knot exterior are conjugate. Hence, if $\rho: \pi_1(S^3\backslash K)\twoheadrightarrow G$ is a maximal rank Coxeter quotient, then the reflections of $G$ form a single conjugacy class. The groups in Theorem \ref{CoxeterClassification} that meet these criteria are $A_3$, $A_4$, $A_5$, $D_4$, $D_5$, $H_3$, and $H_4$. However, to be sure that our search is exhaustive, we need to take into account all Coxeter {\it presentations} of finite Coxeter groups, since, in order to detect the existence of a surjective homomorphism onto $G$, we need to consider all possible minimal generating sets for $G$ within the specified conjugacy class of reflections.

\begin{thm}
Given a finite Coxeter group $C(\Gamma)$ such that the reflections of $C(\Gamma)$ form a single conjugacy class, then $\Gamma$ is of type $A, B, D, E, F, G, H,$ or $I$.
\end{thm}

This stronger statement follows from the proof of the classification of finite Coxeter groups (see, for example, Sections 2.7 and 6.4 of Humphreys \cite{Humphreys1}). As a consequence of this theorem and our previous observations, every Coxeter presentation of a finite Coxeter group with Coxeter rank 3, 4, or 5 such that the set of reflections form a single conjugacy class is one of $A_3$, $A_4$, $A_5$, $D_4$, $D_5$, $H_3$, and $H_4$. Thus, we can restrict to these presentations when implementing our search for all MRCQs from knot groups to finite Coxeter groups for tabulated knots through 16 crossings.

\subsection{Cappell and Shaneson's Meridional Rank Conjecture}

Recall that a meridian of a link $L$ is a based loop $m: S^1\to  S^3\backslash L$ which is freely homotopic to the boundary of an embedded disk $D^2\hookrightarrow S^3$ intersecting $L$ transversally once. The {\it meridional rank}  $\mu(L)$ is the smallest number of elements of $\pi_1(S^3\backslash L)$ represented by meridians which suffice to generate the group. The {\it bridge number} $\beta(L)$ is the minimal number of local maxima of $L$ with respect to the standard height function $h: \mathbb{R}^3\to \mathbb{R}$, taken over all embeddings $l: \coprod S^1\hookrightarrow \mathbb{R}^3$ isotopic to $L$ for which $h_{|l}$ is Morse. One readily derives from the Wirtinger presentation of $\pi_1(S^3\backslash L)$ that the inequality $\beta(L)\geq \mu(L)$ holds for all links, since the meridians near the local maxima of $L$ are seen to generate the group. Cappell and Shaneson asked if the two invariants in fact coincide.

\begin{conj}[MRC]
Let $L\subset S^3$ be a link. Are its bridge number and meridional rank equal?
\end{conj}

We approach this question by studying an intermediate quantity, the {\it Wirtinger number} of $L$, defined in~\cite{blair2020wirtinger} via a combinatorial procedure we now recall. Let $D$ be a link diagram and let $\text{W}(D)$ be a subset of the set of strands in $D$. We will refer to the elements of $\text{W}(D)$ as colored strands. The data $(D, \text{W}(D))$ represents a {\it partially colored diagram}. Denote by $c$ a crossing in $D$ and by $o$, $u_1$ and $u_2$ the overstrand and the two understrands at $c$. When $\{o, u_1\}\subset \text{W}(D)$ and  $\{u_2\}\notin \text{W}(D)$, we say a {\it coloring move} can be performed at $c$, by setting  $\text{W}'(D):=\text{W}(D)\cup \{u_2\}$. We refer to the partially colored diagram $(D, \text{W}'(D))$ as the result of performing a coloring move on $\text{W}(D)$ at $c$ and we write $\text{W}(D)\longrightarrow \text{W}'(D)$. 

Let $|D|$ denote the number of crossings in $D$. A {\it complete coloring sequence} for $D$ consists of a collection of $n$ strands in $D$, $\{s_1, \dots, s_n\}:=\text{W}_1(D)$, together with $|D|-n$ coloring moves $$\text{W}_1(D)\longrightarrow \text{W}_2(D)\dots \longrightarrow \text{W}_{|D|-n}(D),$$ where ${W}_{|D|-n}(D)$  is the set of all strands in $D$. Each of the initial strands $s_i\in {W}_1(D)$ is called a {\it seed strand} or simply a {\it seed} for the sequence. When a complete coloring sequence exists starting with $\text{W}_1(D)$, we say that  the strands in $\text{W}_1(D)$ are a {\it generating set of seeds} for $D$. 

\begin{defn}\label{def-wirt}
	The {\it Wirtinger number} of a link diagram $D$, denoted $\omega(D)$, is the smallest integer $n$ such that there exist a generating set of seeds for $D$ with $n$ elements. The {\it Wirtinger number} of a link $L$, denoted $\omega(L)$, is the minimal value of $\omega(D)$ over all diagrams $D$ of $L$. 
\end{defn}

The motivation for this definition is straight-forward: the Wirtinger number of a diagram $D$ gives a combinatorial upper bound on the meridional rank of the corresponding link $L$. Indeed, a coloring move at a crossing $c$ corresponds to the fact that, together, the Wirtinger meridians of the overstrand $o$ and of the understrand $u_1$ generate the Wirtinger meridian of the second understrand $u_2$; this is immediate from the Wirtinger relation at $c$. Thus, if there exists a coloring sequence for $D$ starting from a collection of seeds  $\{s_1, \dots, s_n\}$, then the Wirtinger meridians of these strands generate the group of the link $L$ and, therefore, $\mu(L)\leq \omega (D)$. This inequality holds for any diagram $D$ of $L$, showing that, in fact,  $$\mu(L)\leq \omega (L).$$

On the other hand, the argument used previously to show that $\beta(L)\geq \mu(L)$ for any link can be used without modification to show that $\beta(L)\geq \omega(L)$.  Put differently, if the strands containing the local maxima of an embedding are chosen as seeds, a complete coloring sequence can be produced by extending the partial coloring successively at crossings of lower height.  Combining the above observations, we see that for any link $L\subset S^3$, $$\beta(L)\geq \omega(L) \geq \mu(L).$$

\begin{thm}[\cite{blair2020wirtinger}] Let $L\subset S^3$ be a link. Its Wirtinger number and bridge number are equal: $\omega(L)=\beta(L).$ \label{thm-wirt}
\end{thm}

The meridional rank conjecture is thus equivalent to proving the inequality $\omega(L)=\mu(L)$ for all links. 
As outlined in the Introduction, one way to establish this equality is to exhibit a diagram $D$ of a link which admits a Coxeter quotient of rank $\omega(D)$. Our main result, Theorem~\ref{thm-main}, is identifying all diagrams $D$ of knots through 16 crossings whose groups admit quotients of rank $\omega(D)$ onto finite Coxeter groups.  When a knot $K$ has this property, we conclude, as in Proposition~\ref{prop-equalities}, that $$\omega(D)\geq \beta(K)\geq \mu(K)\geq \omega(D).$$  

A Coxeter quotient of a knot group can be described in any diagram of the knot, as reviewed next. Conversely, the existence of a quotient can be diagrammatically detected; see Lemma~\ref{lem-AnyDiagram}. 

\subsection{Knot colorings and knot group quotients}

One of the early methods for distinguishing knots is via Fox $p$-colorings of their diagrams. Let $D$ be a diagram of a link $L$. A $p$-coloring of $D$ is an assignment $f(s)\in \{1, \dots, p\}$ for each strand $s$ in $D$, subject to the condition that at every crossing in $D$ the relation 
\begin{equation} \label{equation-Fox}
	f(u_1)+f(u_2)-2f(o)\equiv 0\mod p
\end{equation} holds, where $o$ is the overstrand and $u_1, u_2$ the two understrands. Let $D_p=\langle x, y| x^2=y^p=1, xyxy=1\rangle$ denote the dihedral group of order $2p$. A Fox $p$-coloring defines a homomorphism $\varphi: \pi_1(S^3\backslash L)\to D_p$ by mapping the Wirtinger meridian $m_s$ of a strand $s$ to a reflection in $D_p$ determined by $f$: $$\varphi[m_s]:=xy^{f(s)}.$$ Given a crossing in $D$, Equation~(\ref{equation-Fox}) guarantees that the Wirtinger relation among the meridians at this crossing is satisfied by the images of these meridians under $\varphi$. The assignment of integers mod~$p$ to each strand in a link diagram $D$ determines a group homomorphism if and only if the equation is satisfied at every crossing in $D$.

When a knot or link admits many distinct Fox $p$-colorings for a fixed $p$, the number of such colorings can be used to derive a lower bound on its meridional rank. However, the existence of a {\it single} homomorphism onto a given dihedral group can only prove that the meridional rank of a knot is bigger than one, since two reflections suffice to generate the image. For the purpose of studying MRC, it is therefore more helpful  to find quotients from knot groups to groups which require many generators in a fixed conjugacy class (for link groups, conjugacy classes). We employ finite Coxeter groups to this end. See Section~\ref{section-coxeter review} for a definition and quick review of the properties and classification of these groups. We will make extensive use the fact that homomorphisms from a link group to any group can be described diagrammatically, just like Fox colorings.

\begin{defn} \label{def-coherent}
Let $G$ be a group and let $D$ be a diagram of an oriented link $L$. Denote by $s(D)$ the set of strands in $D$. A {\it $G$-coloring of $D$} is 
	a map $$r: s(D)\to G$$
	$$s_i\mapsto g_{s_i}$$ such that for any crossing $c$ in $D$ with overstrand $s_i$ and understands $s_j$ and $s_k$, the relation holds: 
	\begin{equation}\label{crossing-relation}
	g_{s_i}g_{s_j}g_{s_i}^{-1}=g.
	\end{equation}
\end{defn}
In order to pass from a $G$-coloring of a diagram of $L$ to a homomorphism $\pi_1(S^3\backslash L)\to G$, we map the Wirtinger meridian of any strand $s$, with the orientation determined by the orientation of $L$, to the element $g_s\in G$.
\begin{lem}
	\label{lem-AnyDiagram}
	Let $D$ be a diagram of a link $L$ and $G$ a Coxeter group. There exists a good quotient  $\varphi: \pi_1(S^3\backslash L)\twoheadrightarrow G$ mapping meridians of $L$ to reflections in $G$ if and only if $D$ admits a $G$-coloring by reflections.
\end{lem} 

\begin{proof}
Let $\mu_{s_i}$ denote the Wirtinger meridian of the strand $s_i$, with the orientation induced by the given orientation on $L$. Given a good quotient $\varphi$, we define a $G$-coloring of $D$ by setting $r(s_i)=\varphi(\mu_{s_i})$. Since $\varphi$ is a good quotient, $\varphi(\mu_{s_i})$ is a reflection for each $s_i\in s(D)$, so $D$ admits a $G$-coloring by reflections.

For the converse, denote by $r: s(D)\to G$ a given $G$-coloring of $D$ by reflections. As above, we define a corresponding map $\rho: \{m_{s} | s \text{ a strand in } D \} \to G$ by setting $\rho(m_{s})=g_{s}$, where $m_{s}$ is the Wirtinger meridian of the srand $s_i$.  Since the Wirtinger meridians in a diagram generate the link group, the assignment $\rho$ extends to a map $\rho: \pi_1(S^3\backslash L)\to G$. As in the case of Fox colorings, Equation~\ref{crossing-relation} guarantees that this map is a homomorphism. All meridians of $L$ are conjugate to Wirtinger meridians, and therefore map to reflections in $G$. Hence, a $G$-coloring of $D$ induces a good quotient $\rho: \pi_1(S^3\backslash L)\to G$.
\end{proof}

This well-known lemma is included to highlight the fact that, if $D$ is a diagram of a link $L$ and $D$ does not admit a coherent labeling by reflections in a Coxeter group $G$, then no homomorphism $\pi_1(S^3\backslash L)\twoheadrightarrow G$ exists mapping meridians of $L$ to reflections. The ability to work with {\it any} diagram of $L$ is a useful counterpoint to results which establish the meridional rank conjecture under the assumption that there exists a diagram with certain preferred properties. 

When a $G$-coloring of $D$ induces a good maximal rank Coxeter quotient to $G$, we say the $G$-coloring of $D$ is a diagrammatic MRCQ. We will be performing exhaustive searches for such quotients, using the following observation. 

\begin{rem}\label{rem:seeds-suffice}
	Let $D$ be a diagram for a link $L$ and let $s$ denote a generating set of seeds for $D$. Assume $L$ admits a quotient $\rho: \pi_1(S^3\backslash L)\to G$. In order to define this quotient, it suffices to determine the images under $\rho$ of Wirtinger meridians of the strands in $s$.This partial coloring will extend uniquely to a $G$-coloring of~$D$. 
\end{rem}

A class of link diagrams admiting natural Artin and Coxeter colorings was discovered by Brunner~\cite{brunner1992geometric}. The corresponding links were later called {\it twisted}. A link $L$ is twisted if it admits a diagram $D$, reduced in the sense of~\cite{brunner1992geometric}, with the following property: checkerboard color the complementary planar regions of $D$ in such a way that the unbounded region is ``white"; view the  ``black" as a union of disks and twisted bands\footnote{Again, the surface is assumed to be reduced, which means that each disk is incident to at least 3 bands (otherwise the disk becomes absorbed into a band) and the crossings in each band have the same sign.}; this surface contains at least one full twist in each band. For example, standard diagrams of pretzel knots are twisted when every parameter of the pretzel is least 2 in absolute value. See Figure~\ref{fig-pretzel} for an example of a twisted diagram. 

Given a twisted link $L$, Brunner showed how to define a quotient of $\pi_1(S^3\backslash L)$ onto an Artin group $G$, by labeling the strands in a twisted diagram $D$ of $L$ with appropriate elements of $G$. In Figure~\ref{fig-pretzel}, a twisted diagram is labeled by elements of an Artin group, following Brunner's construction.  A generating set for the group is in bijection with the planar regions in the complement of the twisted surface determined by $D$. The relations in the group are determined by the number of crossings in each twisted band of $D$.

It is convenient to replace $G$ by its natural Coxeter quotient, where the relation $x^2=1$ is added for each of the Artin generators. This will allow us to disregard orientations (since a reflection is equal to its inverse) and to make use of results like the classification of finite Coxeter groups.

\begin{figure}
\begin{center}\labellist
\small
 \pinlabel $a$ at 220 524
  \pinlabel $a$ at 220 580
   \pinlabel $b$ at 245 526
  \pinlabel $b$ at 245 570
   \pinlabel $c$ at 288 522
  \pinlabel $c$ at 288 574
\endlabellist
\includegraphics[height=2.6in, width=2.3in]{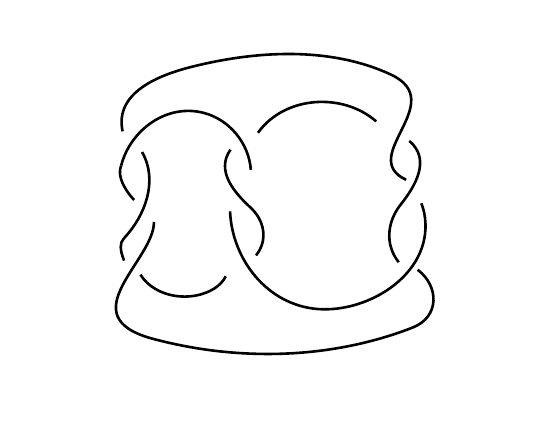}
\caption{The pretzel knot $P(3,-3,3)$, together with a quotient onto the Artin group  $\langle a, b, c |(ab)^3=(bc)^3=(ac)^3=1\rangle$ or 
the Coxeter group $\langle a, b, c | a^2=b^2=c^2= 1, (ab)^3=(bc)^3=(ac)^3=1\rangle$. 
\label{fig-pretzel}
}
\end{center}
\end{figure}
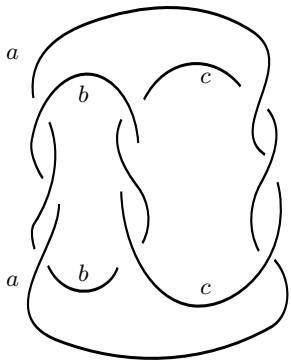

For a link $L$ presented in a twisted diagram $D$, in order to describe a quotient of the group of L, it suffices to determine the images of the two meridians at one end of each twist region in $D$. The images of the remaining meridians under this quotient will be determined by the Wirtinger relations at crossings. Again, refer to  Figure~\ref{fig-pretzel} for an explicit example of this general principle. Brunner's idea is to assign matching generators at the two ends of every twist region. This forces certain Coxeter relations in the quotient, determined by the number of crossings in each of the twist regions. 

We now turn to the rank of the quotient in Brunner's construction. As previously noted, Coxeter generators are in bijection with the regions in the complement of a twisted surface bounded by $L$. Thus, the number of such regions is a lower bound for the meridional rank of $L$. Using the Wirtinger number, matching upper bounds on the bridge numbers were found, which proved MRC for twisted links~\cite{baader2021coxeter}. A similar technique was applied beyond those links for which Brunner found diagrammatic Coxeter quotients, for example to Montesinos links and other natural infinite families of arborescent links~\cite{baader2021coxeter, baader2020twigs}, proving the MRC in these cases as well.

Two-bridge knots are a class of examples which illustrate the limitations of working with twisted diagrams. As noted above, the meridional rank of a 2-bridge knot is always detected in a maximal Coxeter quotient, namely by a quotient onto a dihedral group. However, checkerboard-coloring the diagram of a two-bridge knot produces a twisted surface in few cases. An exhaustive search for Coxeter quotients can prove more effective in practice than one which relies on the existence of a diagram with certain favorable properties.

\section{Homomorphism search}\label{section-homsearch}

\subsection{Summary of the algorithm}
The algorithm used for obtaining our results takes as input the Gauss code $G$ of a knot diagram $D_G$ representing a knot $K_G$. Following the algorithm developed in~\cite{blair2020wirtinger} and available at~\cite{paul2018}, the Gauss code is translated into the following data associated to $D_G$: a set of strands, denoted $S_G=\{s_1,s_2,...,s_j,...,s_n\}$; and a set of crossings, denoted  $C_G=\{c_1,c_2,...,c_j,...,c_m\}$. Next, the algorithm from~\cite{paul2018} is run to calculate the Wirtinger number of $D_G$ and to identify a minimal set of seed strands $E_G=\{e_1,\;e_2,...,\;e_j,...,\;e_{ \omega (D_G)}\}$. See Figure~\ref{fig-Gauss} for an example of how the seeds strands in $D_G$ are recorded in terms of $G$.

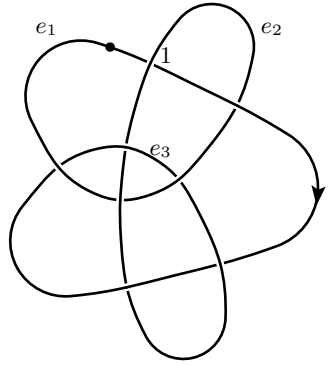
\begin{figure}
\labellist
\small
 \pinlabel $\bullet$ at 87 123
 \pinlabel $1$ at 108 120
 \pinlabel $e_1$ at 63 130
 \pinlabel $e_2$ at 148 130
 \pinlabel $e_3$ at 106 84
\endlabellist
\includegraphics[width=3in, height=2in]{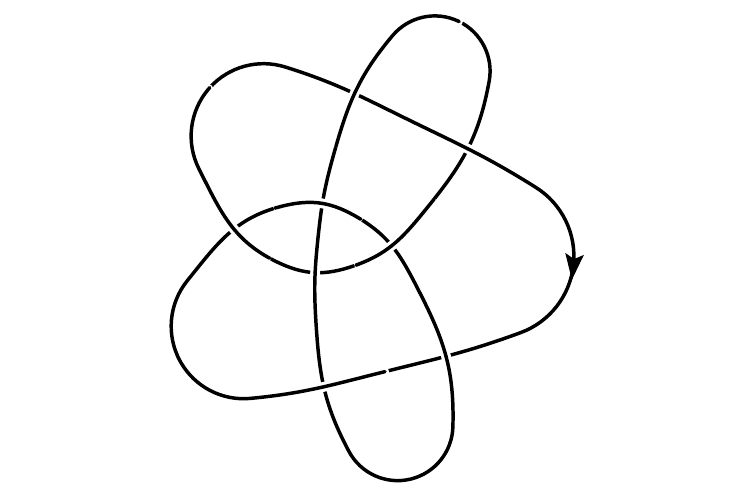}
\caption{A diagram of the knot $8_{16}$ with Gauss code $\{-1, 2, -3,4, -8, 6, -7, 3, -4,$ $5, -6, 1, -2, 7, -5, 8\}$ and seed strands $\{e_1, e_2, e_3\} = \{(-8,5,-1), (-6,1,-2), (-8,6,-7)\}$.} \label{fig-Gauss}
\end{figure}

Let $r_1, r_2,...,r_n$ be a minimal generating set of reflections for a Coxeter group $H$ with Coxeter rank $n=\omega (D_G)$ for some fixed knot diagram $D_G$. As observed in Remark~\ref{rem:seeds-suffice}, coloring the strands in $E_G$ by elements of $H$ suffices to determine the image of all of $\pi_1(S^3\backslash K_G)$ under a (potential) homomorphism to $H$. Fix a bijective map from $r_1, r_2,...,r_n$ to $E_G$. Since $E_G$ is a generating set of seeds,  by repeatedly applying the Wirtinger relations at crossings, this partial coloring can be extended to an assignment of reflections in $H$ to all strands of $D_G$. 
In case this assignment constitutes a coherent $H$-coloring of $D_G$, there exists a maximal rank Coxeter quotient from the knot group $K_G$ to $H$. By Proposition~\ref{prop-equalities}, such a homomorphism implies that $\omega (D_G)$ is equal to the meridional rank of $K_G$ and to the bridge number of $K_G$. 

Given a finite Coxeter group $H$, let $R(H)$ denote the set of all reflections in $H$ and let $Gen(H)$ be the set of all minimal generating sets of reflections for $H$. If $\omega (D_G)$ equals the Coxeter rank $r(H$), a brute force method of searching for good quotients for $K_G$ to $H$ would be to check, for every set $R$ in $Gen(H)$, whether every possible bijection from $R$ to $E_G$ extends to an $H$-coloring of $D_G$.  This can be done as follows. Start with a bijection from $R$ to a generating set of seeds in $D_G$, and sequentially extend this partial $H$-coloring using the Wirtinger relations at crossings. Since the process started with a generating set of seeds, it is guaranteed to result in assigning an element of $H$ to each strand in $D_G$. Once every strand of $D$ has been labeled by an element of $H$, check whether the images of the strands under this potential homomorphism satisfy the Wirtinger relations at those crossings which have not been used to create the coloring. When this is the case, the initial assignment defines a quotient from $K_G$ to $H$. See Figure~\ref{fig-fig8}.  

\begin{figure}
\labellist
\small
 \pinlabel $(12)$ at 222 600
 \pinlabel $(23)$ at 261 600
 \pinlabel $(12)$ at 358 600
 \pinlabel $(23)$ at 396 600
 \pinlabel $(13)$ at 300 495
 \pinlabel $(12)$ at 290 600
 \pinlabel $(23)$ at 329 600
 \pinlabel $(13)$ at 367 495
 \pinlabel $(23)$ at 367 360
\endlabellist
\includegraphics[width=14cm, height=3.8 cm]{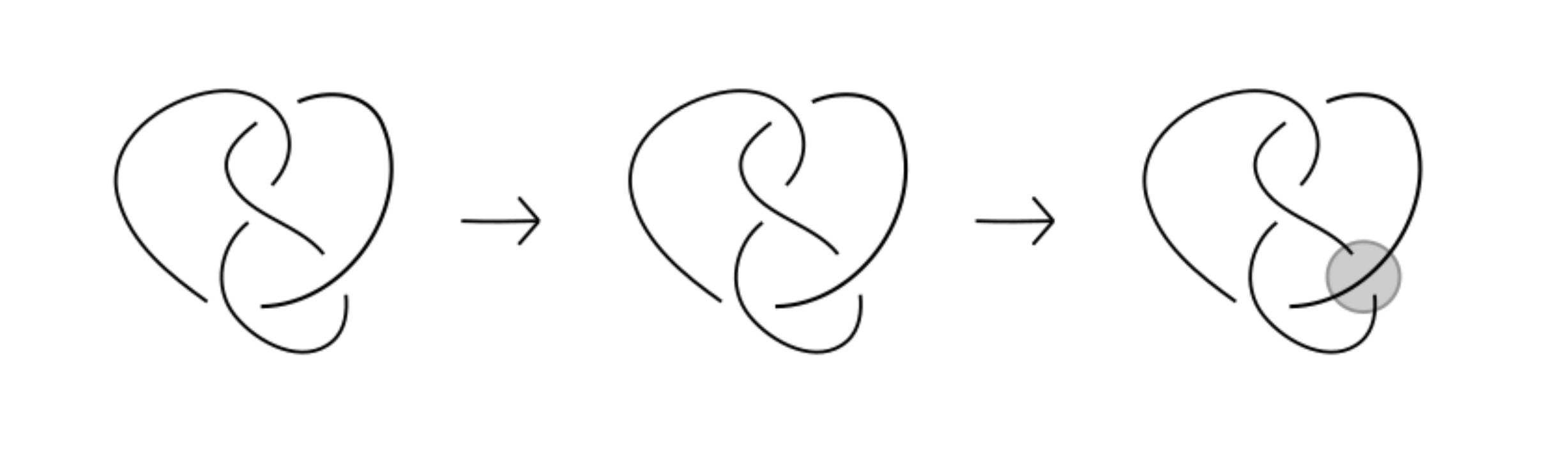}
\caption{Left: a bijection between a pair of seed strands for the given diagram and the generating set $\{(12), (23)\}\subset S_3$. Middle: extending the labeling by applying the Wirtinger relation at a crossing. Right: extending the labeling at a second crossing, then performing a check. The original bijection does not extend to a homomorphism from the group of the Figure-8 knot to $S_3$, as evidenced by the shaded crossing.} \label{fig-fig8}
\end{figure}
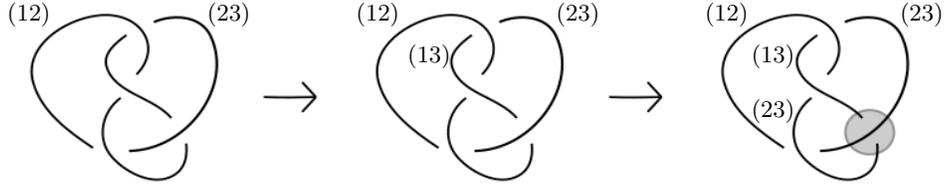

However, this is computationally intensive for the larger finite Coxeter groups. It is also highly redundant because, in general, many generating sets will be related by inner automorphisms of the group. Taking this redundancy into account, for larger Coxeter groups $H$ we implemented a preprocessing step in which we found a smaller subset of $Gen(H)$ which suffices for an exhaustive search. 

Given a Coxeter group $H$, define an equivalence relation on the set $Gen(H)$ by declaring $$\{r_1,r_2,...,r_{r(H)}\} \sim \{\rho_1,\rho_2,...,\rho_{r(H)}\}$$ if there exists $g\in H$ such that $\{g^{-1}r_1 g,g^{-1} r_2 g,...,g^{-1} r_{r(H)} g\}= \{\rho_1,\rho_2,...,\rho_{r(H)}\}$. We say a subset $A\subset Gen(H)$ is \emph{robust} if it contains at least one element from each equivalence class corresponding to the relation~``$\sim$". 

\begin{lem}\label{robust}
Let $G$ be Gauss code for a knot diagram $D$, $E_G$ a minimal set of seed strands for $D_G$, and $A\subset Gen(H)$ a robust set for a Coxeter group $H$. There exists a diagrammatic MRCQ of $D_G$ onto $H$ if and only if there exists $\{\rho_1, \rho_2,...,\rho_n\}\in A$ and a bijection of $\{\rho_1, \rho_2,...,\rho_n\}$ to $E_G$ that can be extended to an $H$-coloring of $D_G$.
\end{lem}

\begin{proof}

Suppose there exists $\{\rho_1, \rho_2,...,\rho_n\}\in A$ and a bijection of $\{\rho_1, \rho_2,...,\rho_n\}$ to $E_G$ that can be extended by repeatedly applying the Wirtinger relations at each crossing to an $H$-coloring of $D_G$. By Lemma \ref{lem-AnyDiagram}, there exists a maximal rank Coxeter quotient of $D_G$ to $H$. 

For the converse, suppose there exists a diagrammatic MRCQ $\phi$ of $D_G$ to $H$. Denote the elements of $E_G$ by $\{e_1,\;e_2,...,\;e_j,...,\;e_{ \omega (D_G)}\}$. By definition, $\phi$ is a good quotient so it maps $\{e_1,\;e_2,...,\;e_j,...,\;e_{ \omega (D_G)}\}$ to a set of reflections $\{r_1,\;r_2,...,\;r_j,...,\;r_{ \omega (D_G)}\}$ in $H$. Since $\phi$ is of maximal rank, $\omega (D_G)=r(H)$. Since $\{e_1,\;e_2,...,\;e_j,...,\;e_{ \omega (D_G)}\}$ is a generating set of seeds for $D_G$, the corresponding Wirtinger meridians form a generating set for the knot group. Consequently, $\{\phi(e_1),\;\phi(e_2),...,\;\phi(e_{ \omega (D_G)})\}$ generates the image of $\phi$. Since $\omega (D_G)=r(H)$, it follows that   
 $\{r_1,\;r_2,...,\;r_j,...,\;r_{ \omega (D_G)}\}$  
is a minimal generating set of reflections for $H$. Hence, $\{r_1,\;r_2,...,\;r_j,...,\;r_{ \omega (D_G)}\}\in Gen(H)$. Since $A$ is a robust subset of $Gen(H)$, then there exists $\{\rho_1,\;\rho_2,...,\;\rho_j,...,\;\rho_{ \omega (D_G)}\}\in A$ such that $\{\rho_1,\;\rho_2,...,\;\rho_j,...,\;\rho_{ \omega (D_G)}\}\sim \{r_1,\;r_2,...,\;r_j,...,\;r_{ \omega (D_G)}\}$. In particular, there exists $g\in H$ such that $\{g^{-1}r_1g,\;g^{-1}r_2g,...,\;g^{-1}r_{ \omega (D_G)}g\}=\{\rho_1,\;\rho_2,...,\;\rho_j,...,\;\rho_{ \omega (D_G)}\}$. If $\theta$ is the inner automorphism of $H$ given by conjugation by $g$, then $\theta \circ \phi$ is a maximal rank Coxeter quotient of $D_G$ to $H$ and there is a bijection from $\{\rho_1,\;\rho_2,...,\;\rho_j,...,\;\rho_{ \omega (D_G)}\}$ to $E_G$ that can be extended by repeatedly applying the Wirtinger relations at each crossing to an $H$-coloring of $D_G$.
\end{proof}

By Lemma~\ref{robust}, in order to verify the existence of a MRCQ of $D_G$ to $H$, it suffices to check whether any bijection from a set in $A$, a robust subset of $Gen(H)$, to a minimal set of seed strands of $D_G$  can be extended to an $H$-coloring of $D_G$. 

Given Gauss code $G$ of a knot diagram $D_G$ representing a knot $K_G$, we implemented the following steps to perform an exhaustive search for good homomorphisms from $D_G$ to a finite Coxeter group $H$.

\begin{enumerate}
\item $D_G$ is parsed into a set of strands $S_G$ and the algorithm from~\cite{paul2018} is used to find a minimal set of seeds $E_G=\{e_1,\;e_2,...,\;e_j,...,\;e_{ \omega (D_G)}\}\subset S_G$.
\item If $\omega (D_G)=r(H)$ and $A\subset Gen(H)$ is \emph{robust}, then for every $R\in A$ and every bijection $f:R\rightarrow E_G$ we test whether $f$ can be extended to an $H$-coloring of $D_G$.
\end{enumerate}

We can now prove the main result of this paper.

 \begin{proof}[Proof of Theorem~\ref{thm-main}]

Let $G$ be the Gauss code for a diagram $D_G$ of a knot. Let $H$ be a Coxeter group such that the Coxeter rank of $H$ is $\omega (D_G)$ and $D_G$ has a maximal rank Coxeter quotient $\rho$ to $H$. We need to verify that $\rho$ will be detected by our search. By Lemma \ref{lem-AnyDiagram}, the homomorphism $\rho$ induces an $H$-coloring of $D_G$. Applying our algorithm to $G$, we find a minimal set of seed strands $E_G=\{e_1,\;e_2,...,\;e_j,...,\;e_{ \omega (D_G)}\}\subset S_G$. The above $H$-coloring of $D_G$ induces a labeling of $E_G$ by reflections $\{r_1,\;r_2,...,\;r_j,...,\;r_{ \omega (D_G)}\}$. Since $E_G$ is a generating set for $\pi_1(S^3\setminus K)$ and the $H$-coloring of $D_G$ induces a MRCQ  from $D_G$ to $H$, 
 we know that $\{r_1,\;r_2,...,\;r_j,...,\;r_{ \omega (D_G)}\}$ is a generating set of $H$.  Given $\mathcal{A}\subset Gen(H)$ a robust set, there exists an element $\{\rho_1,\;\rho_2,...,\;\rho_j,...,\;\rho_{ \omega (D_G)}\}\in \mathcal{A}$ and a $g\in H$ such that $\{\rho_1,\;\rho_2,...,\;\rho_j,...,\;\rho_{ \omega (D_G)}\}= \{g^{-1}r_1g,\;g^{-1}r_2g,...,\;g^{-1}r_jg,...,\;g^{-1}r_{ \omega (D_G)}g\}$. Note that conjugating each label of $D$ by $g$ gives a new $H$-coloring of $D_G$  such that $E_G$ is labeled by $\{\rho_1,\;\rho_2,...,\;\rho_j,...,\;\rho_{ \omega (D_G)}\}$. Therefore, in its search through all bijections of the form $f:E_G\rightarrow A$, where $A$ is an element of the robust set $\mathcal{A}$, the algorithm will find a labeling of $E_G$ by the generators $\{\rho_1,\;\rho_2,...,\;\rho_j,...,\;\rho_{ \omega (D_G)}\}$ which induces an $H$-coloring of $D_G$. Thus, the algorithm will return a positive hit for a MRCQ to $H$. 

By construction, if the algorithm returns a positive hit for a maximal rank Coxeter quotient to $H$, then the Coxeter rank of $H$ is $\omega (D_G)$ and there is an $H$-coloring of $D_G$.
\end{proof}

\subsubsection{MRCQs to finite Coxeter groups among all knots up to 16 crossings}

We implemented the algorithm described above in Python and searched for all maximal rank Coxeter qutients to finite Coxeter groups among the $1,701,936$ prime knots~\cite{HTW} of crossing number less than or equal to 16. It was shown in~\cite{paul2018}  that all Gauss codes available in the census of these knots result in diagrams with Wirtinger number at most 5. Therefore, we designed the code to search for good homomorphisms to finite Coxeter groups of Coxeter rank at most 5. Moreover, the meridians of a knot group form a single conjugacy class. Additionally, the MRC is known for knots of Wirtinger number two~\cite{BZ89}, and generating suitable homomorphisms for the knot group of Wirtinger number two knots to dihedral groups is well-understood. As a result, our code searches for maximal rank Coxeter quotients to those finite Coxeter groups of Coxeter rank 3, 4 or 5 whose reflections constitute a single conjugacy class. As discussed in Section \ref{Coxeterbackground}, every Coxeter presentation for such a group is one of $A_3$, $A_4$, $A_5$, $D_4$, $D_5$, $H_3$, and $H_4$. In Section \ref{Sec:robust}, we discuss how we generated robust sets of generating sets for each of these groups.

\subsection{Generating Robust Sets}\label{Sec:robust}

To illustrate our approach, we outline the process we used to generate a robust set of generating sets for the group $D_4$. First, we represented $D_4$ as a subgroup of $GL_4(\mathbb{R})$, the general linear group of degree four. Specifically, $D_4$ is isomorphic to the subgroup generated by the following matrices:
$$ \begin{bmatrix}
-1 & 1 & 0 & 0\\
0 & 1&  0 & 0 \\
0 & 0 &  1 & 0 \\
0 & 0 &  0 & 1 
\end{bmatrix},  \begin{bmatrix}
1 & 0 & 0 & 0\\
1 & -1&  1 & 1 \\
0 & 0 &  1 & 0 \\
0 & 0 &  0 & 1 
\end{bmatrix}, \begin{bmatrix}
1 & 0 & 0 & 0\\
0 & 1 &  0 & 1 \\
0 & 1 &  -1 & 0 \\
0 & 0 &  0 & 1 
\end{bmatrix}, \begin{bmatrix}
1 & 0 & 0 & 0\\
0 & 1 &  0 & 1 \\
0 & 0 &  1 & 0 \\
0 & 1 &  0 & -1 
\end{bmatrix}$$
Since the reflections in $D_4$ are all contained in a single conjugacy class, we know that there exists a robust set for the group $D_4$ so that every generating set in the robust set contains the matrix $A=\begin{bmatrix}
-1 & 1 & 0 & 0\\
0 & 1&  0 & 0 \\
0 & 0 &  1 & 0 \\
0 & 0 &  0 & 1 
\end{bmatrix}$. Since $D_4$ contains a total of 12 reflections, if we fix the first reflection, this leaves $11\times 10\times 9=990$ sets of four distinct reflections that contain $A$. Recall that each reflection in $\mathbb{R}^4$ has a rank 1 eigenspace corresponding to the eigenvalue $-1$ and a rank 3 eigenspace corresponding to the eigen value $1$. For a set of four reflections to generate $D_4$ is must be that the set of four eigenvectors corresponding to the four one-dimensional eigenspaces associated to the $-1$ eigenvalues for each of the four reflections must span $\mathbb{R}^4$. We determined by direct computation that $630$ of the $990$ sets of reflections had this property.  We found that $624$ of these sets generated a group of order $192=|D_4|$ and that the remaining 6 groups generated $(\mathbb{Z}/2\mathbb{Z})^4$. Thus, this set of $624$ four-element sets of reflections is a robust set of generating sets for the group $D_4$. Naturally, smaller robust sets help reduce run time when searching for MRQCs across millions of knot diagrams. 

Building robust sets of generating sets for each of the groups $A_3$, $A_4$, $A_5$, $D_5$, $H_3$, and $H_4$ was done by an analogous method  For each group, computational resources devoted to building a small robust set were balanced against computational time saved by running the maximal rank Coxeter quotient search algorithm using a smaller robust set. For example, significant computational time was spent to generate small robust sets for $H_4$, and $D_5$. As in the example outlined above, we started by generating all sets of four (resp. 5) reflections containing a fixed preferred reflection. These sets were trimmed in two ways: first, sets that generated a proper subgroup were removed. Then, we implemented a brute force search that identified when two generating sets were related by an inner automorphism and deleted one of the redundant generating sets. Ultimately, we found a robust set of generating sets for $H_4$ that contained $25,224$ elements, down from $11,703,240$ sets before the trimming process, and a robust set of generating sets for $D_5$ that contained $1,778$ elements, down from $1,860,480$ sets before the trimming process. All robust sets generated are available here \cite{nate2022}.

\section{Computational findings}\label{section-tables}
In this section we organize our computational data into tables. Note that we found maximal rank Coxeter quotients for $595,515$, roughly $35\%$,  
of all $1,696,390$ knots of crossing number at most 16 that have minimal diagrams with Wirtinger number 3, 4 or 5. It is important to note that, the Wirtinger number detects all 5546 2-bridge knots of crossing number at most 16 and all knots with crossing number at most 16 have Wirtinger number at most 5~\cite{blair2020wirtinger}. Moreover, all 2-bridge knots admit a maximal rank Coxeter quotient to a finite dihedral group. Hence, exactly  $601,061$ of the $1,701,936$ prime knot diagrams with crossing number at most 16 admit a maximal rank Coxeter quotient to a finite Coxeter group. 

This work also verifies the MRC for a large portion of tabulated knots. Already, knots with diagrams of Wirtinger number 2 and 3 were known to satisfy the MRC~\cite{BZ89}. In addition, this computation establishes the MRC for $227,163$ knots with Wirtinger number 4 or 5. Altogether, the MRC has been verified for at least $1,363,137$, or approximately $80.1\%$, of all $1,701,936$ prime knots with crossing number at most 16. 

\begin{table}[ht]
\caption{Knots with maximal rank Coxeter quotients by group type}
\centering
\begin{tabular}{||c c c c c||} 
 \hline
 Crossing number & Prime knots & MRCQ to $A_3$, $A_4$, or $A_5$ & MRCQ to $D_4$, or $D_5$ & MRCQ to $H_3$, or $H_4$ \\ [0.5ex] 
 \hline\hline
 3 & 1 & 1 & 0 & 0 \\ 
 \hline
 4 & 1 & 0 & 0 & 0 \\
 \hline
 5 & 2 & 0 & 0 & 0 \\
 \hline
 6 & 3 & 1 & 0 & 0 \\
 \hline
 7 & 7 & 2 & 0 & 0 \\
 \hline
 8 & 21 & 7 & 0 & 0 \\
 \hline
 9 & 49 & 17 & 0 & 9 \\
 \hline
 10 & 165 & 39 & 0 & 40 \\
 \hline
 11 & 552 & 121 & 15 & 124 \\
 \hline
 12 & 2176 & 370 & 13 & 537 \\
 \hline
 13 & 9988 & 1772 & 316 & 2572 \\
 \hline
 14 & 46972 & 7069 & 1099 & 12494 \\
 \hline
 15 & 253293 & 37490 & 7997 & 66962 \\
 \hline
 16 & 1388705 & 183509 & 457923 & 363456 \\
 \hline
 Totals & 1701936 & 230398 & 55233 & 446194 \\ [1ex]
 \hline
\end{tabular}
\end{table}

\begin{table}[ht]
\caption{Prime knots with $\omega(D)=3$ which admit maximal rank Coxeter quotients}
\centering
\begin{tabular}{||c c c c c||} 
 \hline
 Crossing number & Knots with $\omega(D)=3$ & MRCQ to $A_3$ & MRCQ to $H_3$ & MRCQ to $A_3$ or $H_3$ \\ [0.5ex] 
 \hline\hline
 3 & 0 & 0 & 0 & 0 \\ 
 \hline
 4 & 0 & 0 & 0 & 0 \\
 \hline
 5 & 0 & 0 & 0 & 0 \\
 \hline
 6 & 0 & 0 & 0 & 0 \\
 \hline
 7 & 0 & 0 & 0 & 0 \\
 \hline
 8 & 9 & 6 & 0 & 6 \\
 \hline
 9 & 24 & 8 & 9 & 16 \\
 \hline
 10 & 120 & 26 & 40 & 64 \\
 \hline
 11 & 446 & 85 & 109 & 190 \\
 \hline
 12 & 1952 & 312 & 489 & 729 \\
 \hline
 13 & 8614 & 1221 & 1995 & 2954 \\
 \hline
 14 & 39291 & 5495 & 8808 & 13104 \\
 \hline
 15 & 187121 & 25181 & 41771 & 61343 \\
 \hline
 16 & 892851 & 116071 & 198290 & 288557 \\
 \hline
 Totals & 1130428 & 148405 & 251511 & 366963 \\ [1ex]
 \hline
\end{tabular}
\end{table} 

\begin{table}[ht]
\caption{Prime knots with $\omega(D)=4$ and maximal rank Coxeter quotients}
\centering
\begin{tabular}{||c c c c c c||} 
 \hline
 Crossing $\#$ &  $\omega(D)=4$ & MRCQ to $A_4$ & MRCQ to $H_4$ & MRCQ to $D_4$ & To $A_4$, $H_4$ or $D_4$ \\ [0.5ex] 
 \hline\hline
 3 & 0 & 0 & 0 & 0 & 0 \\ 
 \hline
 4 & 0 & 0 & 0 & 0 & 0 \\
 \hline
 5 & 0 & 0 & 0 & 0 & 0 \\
 \hline
 6 & 0 & 0 & 0 & 0 & 0 \\
 \hline
 7 & 0 & 0 & 0 & 0 & 0 \\
 \hline
 8 & 0 & 0 & 0 & 0 & 0 \\
 \hline
 9 & 0 & 0 & 0 & 0 & 0 \\
 \hline
 10 & 0 & 0 & 0 & 0 & 0 \\
 \hline
 11 & 15 & 15 & 15 & 15 & 15 \\
 \hline
 12 & 48 & 13 & 48 & 13 & 48 \\
 \hline
 13 & 1022 & 456 & 577 & 316 & 595 \\
 \hline
 14 & 6958 & 1387 & 3686 & 1069 & 3788 \\
 \hline
 15 & 64723 & 11944 & 25191 & 7975 & 29588 \\
 \hline
 16 & 488032 & 63258 & 165166 & 42282 & 189566 \\
 \hline
 Totals & 560798 & 77073 & 194683 & 51670 & 223600 \\ [1ex]
 \hline
\end{tabular}
\end{table} 

\begin{table}[ht]
\caption{Prime knots with $\omega(D)=5$ and maximal rank Coxeter quotients}
\centering
\begin{tabular}{||c c c c c||} 
 \hline
 Crossing number & Knots with $\omega(D)=5$ & MRCQ to $A_5$ & MRCQ to $D_5$ & MRCQ to $A_5$ or $D_5$ \\ [0.5ex] 
 \hline\hline
 3 & 0 & 0 & 0 & 0 \\ 
 \hline
 4 & 0 & 0 & 0 & 0 \\
 \hline
 5 & 0 & 0 & 0 & 0 \\
 \hline
 6 & 0 & 0 & 0 & 0 \\
 \hline
 7 & 0 & 0 & 0 & 0 \\
 \hline
 8 & 0 & 0 & 0 & 0 \\
 \hline
 9 & 0 & 0 & 0 & 0 \\
 \hline
 10 & 0 & 0 & 0 & 0 \\
 \hline
 11 & 0 & 0 & 0 & 0 \\
 \hline
 12 & 0 & 0 & 0 & 0 \\
 \hline
 13 & 0 & 0 & 0 & 0 \\
 \hline
 14 & 30 & 30 & 30 & 30 \\
 \hline
 15 & 62 & 22 & 22 & 22 \\
 \hline
 16 & 5072 & 3479 & 3511 & 3511 \\
\hline
 Totals & 5164 & 3531 & 3563 & 3563 \\ [1ex] 
 \hline
\end{tabular}
\end{table} 

\section{Brief remarks}

\subsection{Bridge number and crossing number}
Our computations suggest the following relationship between the bridge number and crossing number of a knot. 
\begin{conj}
	Let $n\geq 3$ and let $K$ be a prime knot with bridge number equal to $n$. The crossing number of $K$ is at least $3n-1$. 
\end{conj}

For all prime knots through 16 crossings, the conjecture can be verified using the upper bounds on the bridge number obtained from the Wirtinger numbers of crossing-number minimizing diagrams in the knot table. Remark also that the lower bound we propose is optimal: for any $n\geq 3$ there exists a knot with exactly $3n-1$ crossings and bridge number $n$, namely the pretzel knot $P(2,3,3,...,3)$. Of course, the conjectured inequality would not hold for links as, for example, an unlink on more than one component would violate it. Non-prime knots also easily violate the inequality, for example the connected sum of a trefoil with itself.

\subsection{Homomorphisms to infinite Coxeter groups}

As previously discussed, maximal rank Coxeter quotients were used in~\cite{baader2021coxeter, baader2020twigs} to prove the Meridional Rank Conjecture for large infinite families of links. The Coxeter quotients used in that proof are in the vast majority of cases infinite. We expect that for a sizable fraction of the knots studied in \cite{baader2021coxeter, baader2020twigs} no maximal rank {\it finite} Coxeter quotients exist, though no large-scale computations have been performed due to the high crossing numbers of these knots and the absence of a tabulation. Nevertheless, we posit that extending the current work to infinite Coxeter groups is likely to result in computing the meridional rank of many more knots. As far as we know, it is an open question whether the meridional rank of a knot is always detected in a finite quotient (not necessarily to a Coxeter group).

We give an explicit example of a 12-crossing knot whose meridional rank is detected in an infinite Coxeter quotient but not in a finite one.

\begin{figure}
\begin{center}\labellist
\small
 \pinlabel $a$ at 196 568
   \pinlabel $b$ at 394 607
  \pinlabel $c$ at 454 606
\endlabellist
\includegraphics[height=2.3in, width=4.1in]{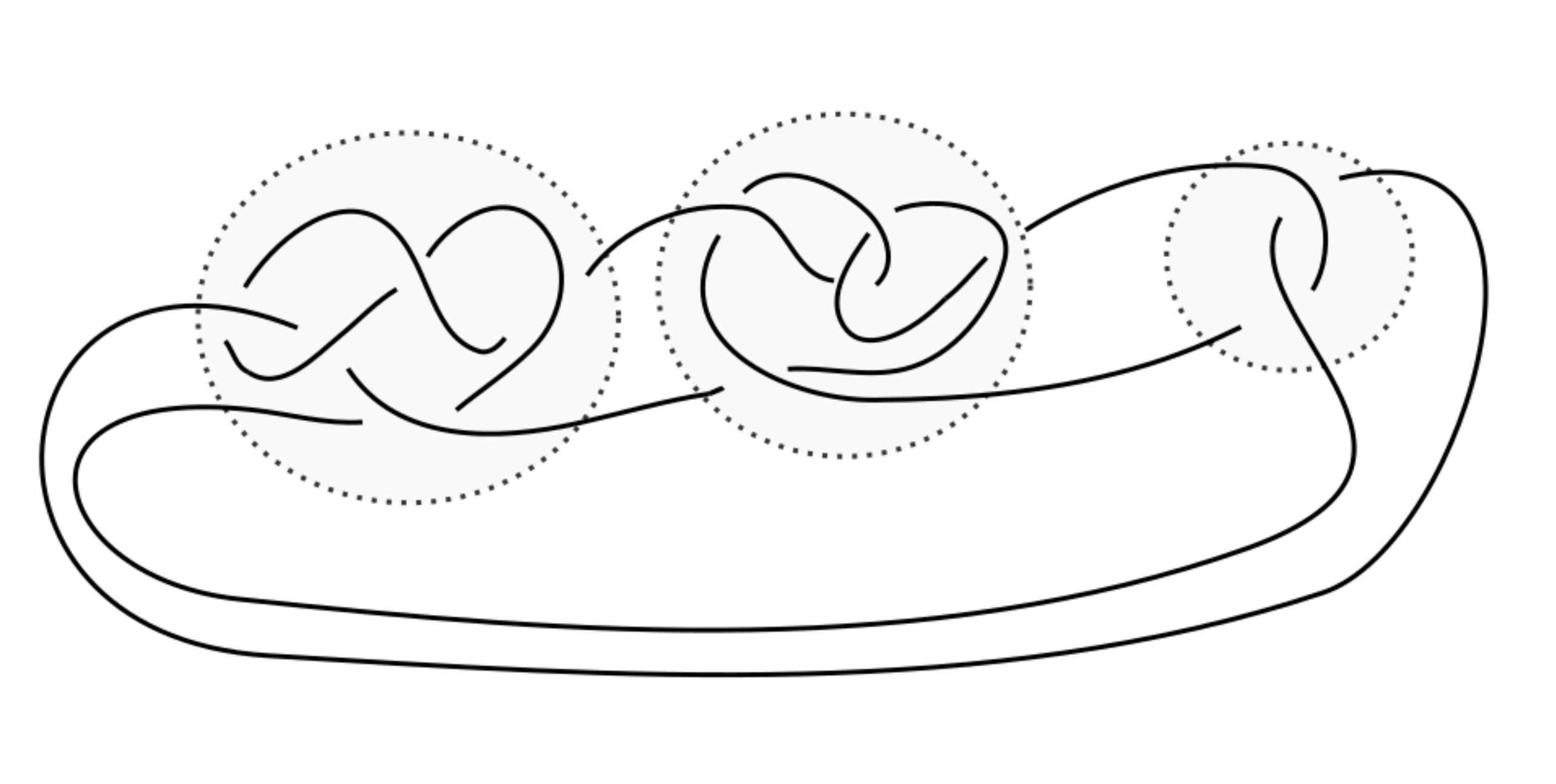}
\caption{The Montesinos knot $12a210$, together with a maximal rank quotient onto the Coxeter group $\langle a, b, c | a^2=b^2= c^2=1,(ab)^3= (ac)^7=(bc)^2=1\rangle$. The knot does not admit a MRCQ onto a finite Coxeter group.
\label{fig-12a210}
}
\end{center}
\end{figure}

\begin{example}
The knot $12a210$ admits a maximal rank Coxeter quotient to the infinite Coxeter group determined by the Coxeter matrix $\begin{bmatrix} 1 & 2 & 3 \\ 2&1&7 \\ 3&7&1
\end{bmatrix},$ see Figure~\ref{fig-12a210}. 
This is a Montesinos knot on three rational tangles, and the existence of this maximal rank Coxeter quotient also follows from~\cite[p. 1551, footnote 2]{baader2021coxeter}.  In contrast, our algorithm proves that there does not exist a homomorphism onto $A_3$ or $H_3$. In other words, even though there is a homomorphism to an infinite Coxeter group that establishes the MRC for this knot, there is no maximal rank Coxeter quotient to a finite Coxeter group that achieves the same.  Many other explicit examples of infinite MRCQs for 3-bridge knots can be found in~\cite{Ryffel2019}.	 
\end{example}

The non-existence of a maximal rank Coxeter quotient for some knots may guide the search for potential counter-examples to the Meridional Rank Conjecture. At a minimum, many knots are ruled out as possible counter-examples by our method. However, it is also known that the meridional rank of a knot is not always detected in a Coxeter quotient. It was shown by Ryffel that many torus knots not only do not admit a maximal rank Coxeter quotient but do not admit any nontrivial Coxeter quotients whatsoever.

\begin{thm}[\cite{Ryffel2019}] Let $p, q\in\mathbb{Z}$ be coprime odd integers such that $p\geq 3$ and $q$ has no factor less than or equal to $\max \{5, p\}$. 	Then the $(p, q)$-torus knot does not admit any non-trivial Coxeter quotients.
\end{thm}

On the other hand, the Meridional Rank Conjecture holds for all torus links~\cite{rost1987meridional}. \\

\section*{Acknowledgement}
The authors would like to thank Curtis Bennett, John Brevik and Jon McCammond for helpful conversations about Coxeter groups; and Sebastian Baader and Levi Ryffel for offering feedback on a draft of this paper. RB and NM were partially supported by NSF grant DMS-1821254. A portion of this work was completed while AK was a guest at the Max Planck Institute for Mathematics in Bonn. We are grateful to MPIM for its hospitality.

\nocite{RZ87}
\bibliographystyle{plain}
\bibliography{wirtinger1}

\hspace{0.5in}

\address{Department of Mathematics, CSULB\\
Long Beach, CA, 90840, USA}\\

\address {Department of Mathematics, University of Notre Dame\\
Notre Dame, IN, 46556, USA}\\

\address {Department of Physics, University of California\\
Berkeley, CA, 94720, USA}\\

\end{document}